\documentclass[letterpaper,11pt]{article}
\usepackage[T1]{fontenc}
\usepackage[english]{babel}
\usepackage{fullpage}
\usepackage{color,xcolor}
\usepackage{amsmath,amsfonts,amssymb,amsthm,mathtools,dsfont}
\usepackage{graphicx,tikz,pgfplots}
\usepackage{algorithm,algpseudocode}
\usepackage{booktabs}
\usepackage{cite}
\usepackage{hyperref}
\usepackage[capitalize,nameinlink,noabbrev]{cleveref}

\hypersetup{colorlinks, hypertexnames=false, pageanchor=true,
            linkcolor=blue, citecolor={green!50!black}, urlcolor=cyan,
            pdftitle={Sparse factorization of the square all-ones matrix of arbitrary order}
            }

\allowdisplaybreaks
\mathtoolsset{centercolon}
\pgfplotsset{compat=1.18}
\newtheorem{definition}{Definition}
\newtheorem{theorem}{Theorem}
\newtheorem{corollary}[theorem]{Corollary}

\newcommand{\reals}                  {\mathbb R}
\newcommand{\natint}                 {\mathbb N}
\newcommand{\Diag}                   {\mathbb D}
\newcommand{\HB}                     {\mathbb{HB}}
\newcommand{\ones}                   {\mathds{1}}
\newcommand{\zeros}                  {\mathbf{0}}

\newcommand{\tran}                   {^{\mathsf T}}               % matrix transpose
\newcommand{\nnz}                    {\mathrm{nnz}}
\DeclareMathOperator{\diag}          {diag}                       % diag
\DeclareMathOperator{\bdiag}         {blkdiag}                    % block diag

\newcommand{\minimize}               {\text{\upshape minimize}}
\newcommand{\eg}{{\it e.g.}}
\newcommand{\ie}{{\it i.e.}}

\newcommand{\oline}[1]{\mkern 1.5mu\overline{\mkern-1.5mu#1}}
\newcommand{\f}{\boldsymbol{f}}
\newcommand{\cE}{\mathcal{E}}
\newcommand{\cG}{\mathcal{G}}
\newcommand{\cH}{\mathcal{H}}
\newcommand{\cN}{\mathcal{N}}
\newcommand{\cV}{\mathcal{V}}
\newcommand{\Jbar}{\oline{J}}
\newcommand{\kbar}{\oline{k}}
\newcommand{\xbar}{\oline{x}}
\newcommand{\Atilde}{\widetilde{A}}
\newcommand{\Jtilde}{\widetilde{J}}
\newcommand{\xtilde}{\tilde{x}}
\newcommand{\That}{\widehat{T}}

\title{Sparse factorization of the square all-ones matrix \\ of arbitrary order\footnote{Author's final version. Accepted for publication in \textit{SIAM Journal on Matrix Analysis and Applications}.}}
\author{%
    Xin Jiang%
    \thanks{School of Operations Research and Information Engineering, Cornell University.
    Email: \textsf{xjiang@cornell.edu}. Most of the work was completed when XJ was at Lehigh University.}
    \and Edward Duc Hien Nguyen%
    \thanks{Department of Electrical and Computer Engineering, Rice University.
    Email: \textsf{\{en18, cauribe\}@rice.edu}.}
    \and C\'esar A. Uribe\footnotemark[3]
    \and Bicheng Ying%
    \thanks{Google Inc. Email: \textsf{ybc@google.com}.}
}
\date{November 18, 2024}

\begin{document}

\maketitle

\begin{abstract}
    In this paper, we study sparse factorization of the (scaled) square all-ones matrix $J$ of arbitrary order. We introduce the concept of hierarchically banded matrices and propose two types of hierarchically banded factorization of $J$: the reduced hierarchically banded (RHB) factorization and the doubly stochastic hierarchically banded (DSHB) factorization. Based on the DSHB factorization, we propose the sequential doubly stochastic (SDS) factorization, in which~$J$ is decomposed as a product of sparse, doubly stochastic matrices. Finally, we discuss the application of the proposed sparse factorizations to the decentralized average consensus problem and decentralized optimization. 
\end{abstract}

%%%%%%%%%%%%%%%
%%  Section  %%
%%%%%%%%%%%%%%%
\section{Introduction}
We study sparse factorization of the real $n \times n$ matrix $J := \tfrac{1}{n} \ones \ones\tran \in \reals^{n \times n}$; that is, we seek to find a (finite) sequence of matrices $\{W^{(k)}\}_{k=1}^q \subset \reals^{n \times n}$ such that
\begin{equation} \label{eq:W-factor}
    W^{(q)} W^{(q-1)} \cdots W^{(1)} = \frac{1}{n} \begin{bmatrix}
        1 & 1 & \cdots & 1 \\
        1 & 1 & \cdots & 1 \\
        \vdots & \vdots & & \vdots \\
        1 & 1 & \cdots & 1
    \end{bmatrix}.
\end{equation}
This problem finds applications in graph theory, systems and control, decentralized optimization, and other fields \cite{delvenne2009optimal,shi2016finite,nguyen2023graphs}. In this paper, we consider the general case where $n$ is an arbitrary integer and propose several types of sparse factorization.
\par
Previous work on the sparse factorization of $J$, or the all-ones matrix $\Jtilde = nJ = \ones \ones\tran$, can be roughly divided into two categories. The first class considers the case in which all the factors are identical, \ie, $W^{(k)} = W$ for all $k \in [q]$. For example, the binary square root of $\Jtilde$ (when $n=p^2$ for some $p \in \natint_{\geq 2}$) is studied in \cite{curtis2007central}. The De Bruijn matrix, first proposed in \cite{bruijn1946ACP}, serves as the $q$-th root of $\Jtilde$ when $n = p^q$, and has been extensively studied in the literature \cite{delvenne2009optimal,trefois2015binary}. For general $n$, the $g$-circulant binary solutions to $W^q = \Jtilde$ have also been investigated~\cite{wang1982circulant, king1985circulant, ma1987circulant, wu2002circulant}. 
\par
The second class of solutions allows for differing factors of~$J$. Among these solutions include one-peer exponential graphs (when $n=2^q$)~\cite{ying2021exponential}, one-peer hyper-cubes (when $n=2^q$)~\cite{shi2016finite}, $p$-peer hyper-cuboids~\cite{takezawa2023exponential, nguyen2023graphs}, and deformable butterfly (DeBut) matrices~\cite{lin2021deformable}. Remarkably, the butterfly matrices \cite{dao2019learning}, which were originally proposed for more general linear transforms and used in deep neural networks, reduce to one-peer hyper-cubes when we study the factorization~\eqref{eq:W-factor} with $n=2^q$. As extensions of the one-peer hyper-cubes (and butterfly matrices), the $p$-peer hyper-cuboids and the DeBut matrices serve as factorization of $J$ for arbitrary~$n$, and the sparsity of both factors depends on the prime factorization of $n$. In particular, when $n$ is a large primal number, both $p$-peer hyper-cuboids and the DeBut matrices reduce to the fully dense matrix~$J$. Allowing for different factors of~$J$, in general, gives greater control over the sparsity of the factors compared to the case in which all the factors are identical~\cite{nguyen2023graphs}.
\par
In this paper, we consider the general case where $n \in \natint_{\geq 2}$ is an arbitrary integer and study sparse factorization of $J$ in the form
\begin{equation} \label{eq:J-factor}
    J = J_0 A J_0,
\end{equation}
where $J_0 = J_1 \oplus \cdots \oplus J_\tau$ with $J_k := \tfrac{1}{n_k} \ones \ones\tran \in \reals^{n_k \times n_k}$, $k \in [\tau]$. (Here, $\oplus$ denotes the direct sum of two matrices.) Throughout the paper, it is assumed that the partition $n = \sum_{k=1}^\tau n_k$ is given, with conditions that will be specified later (see~\eqref{eq:n-partition}). Factorization~\eqref{eq:J-factor} holds for arbitrary matrix order~$n$ and is inspired by the applications of $J$ in decentralized averaging (and optimization). In decentralized averaging, for example, a group of agents each holds a piece of information and cooperates with other agents to compute a global quantity. The communication between agents is modeled by a graph (or a sequence of graphs) $\cG^{(k)} = (\cV, W^{(k)}, \cE^{(k)})$. If the weight matrices $\{W^{(k)}\}$ satisfy~\eqref{eq:W-factor}, then the \textit{exact} global average is computed in~$q$ communication rounds. In modern application scenarios, agents can be abstracted as high-performance computing (HPC) resources and naturally formed into clusters~\cite{kreutz2014software,bera2017software,ying2021bluefog}; see \Cref{sec:app} for a more detailed discussion. Such clustering structure is captured by the proposed form of factorization~\eqref{eq:J-factor}. The block diagonal matrix $J_0$ models the intra-cluster communication, and each sub-block~$J_k$, $k \in [\tau]$, can be further decomposed as~\eqref{eq:W-factor} into, \eg, $p$-peer hyper-cuboids. In contrast, the $A$-factor models the more expensive inter-cluster communication, and the main focus of this paper is to design \textit{sparse} $A$-factors to reduce the communication overhead across clusters. Sparsity in~$A$ is desirable in decentralized averaging (and optimization) as the communication overhead is related to the total number of nonzeros $\nnz(A)$ as well as the largest node degree $d_\mathrm{max} (A) = \max_i \{\nnz(A_{i,:})\}$, where $A_{i,:}$ is the $i$th row of $A$.
\par
\paragraph{Contributions}
In this paper, we study the form of factorization in~\eqref{eq:J-factor} for arbitrary matrix order~$n$ and propose three types of $A$-factors. In the first two types, the sparse factor $A$ has the so-called \textit{hierarchically banded (HB)} structure, and additional properties of $A$ distinguish these two types of HB factorization: (density) reduced HB and doubly stochastic HB. The third one is called the \textit{sequential doubly stochastic (SDS) factorization} and admits an asymmetric, doubly stochastic factor~$A$, which can be further decomposed as a product of several symmetric, doubly stochastic matrices. When applied to decentralized optimization, the proposed sparse factorizations provide more flexibility to balance communication costs and the total number of communication rounds in a decentralized optimization algorithm.
\par
\paragraph{Notation}
Let $\reals$ denote the set of real numbers (\ie, scalars). Let $\reals^n$ denote the set of $n$-dimensional (column) vectors. Let $\reals^{m \times n}$ denote the set of $m$-by-$n$ real matrices, and let $\Diag^n$ denote the set of $n \times n$ diagonal matrices. The set of natural numbers is denoted as $\natint := \{0,1,2,\ldots\}$, and let $\natint_{\geq r}$ denote the set of natural numbers greater than or equal to $r \in \natint$. For any $n \in \natint_{\geq 1}$, let $[n] := \{1,2,\ldots,n\}$. Let~$\ones$ denote the all-ones (column) vector of compatible size. Let $\|\cdot\|$ denote the Euclidean norm of a vector. The direct sum of two matrices $A \in \reals^{m \times n}$ and $B \in \reals^{p \times q}$ forms the block diagonal matrix $A \oplus B := \bdiag(A, B) \in \reals^{(m+p) \times (n+q)}$.
\par
\paragraph{Outline}
In \Cref{sec:hb-mat}, we propose the notion of hierarchically banded (HB) matrices. \Cref{sec:rhb,sec:dshb} study two types of HB factorization. The sequential doubly stochastic (SDS) factorization is discussed in \Cref{sec:sds}. In \Cref{sec:app}, we present the potential usefulness of these sparse factorizations in decentralized averaging and optimization, and concluding remarks are offered in \Cref{sec:conclusion}.

%%%%%%%%%%%%%%%
%%  Section  %%
%%%%%%%%%%%%%%%
\section{Hierarchically banded matrices} \label{sec:hb-mat}
Factorization of the form~\eqref{eq:J-factor} relies on a partition of $n \in \natint_{\geq 2}$: 
\begin{equation} \label{eq:n-partition}
n = \sum_{k=1}^\tau n_k, \; \text{where} \ \{n_k\}_{k=1}^\tau \subset \natint_{\geq 1}, \ \text{and} \ n_k \geq \sum_{j=k+1}^\tau n_j =: m_k \ \text{for all} \ k \in [\tau-1].
\end{equation} 
Such a partition can be constructed systematically, \eg, via the base-$p$ representation of~$n$ (with $p \in \natint_{\geq 2}$). Overloading the binary representation, we denote the base-$p$ representation of~$n$ as $(i_{\tau-1} i_{\tau-2} \cdots i_1 i_0)_p$, where $\tau = \lfloor \log_p (n) \rfloor + 1$. Then, any integer $n \in \natint_{\geq 2}$ can be written as $n = \sum_{k=1}^\tau n_k$, where $n_k = i_{\tau-k} p^{\tau-k}$. (For those $i_k$'s being zero, they are removed before the construction of the partition.) By construction, the condition $n_k \geq m_k$, for all $k \in [\tau-1]$, directly follows from the property of the base-$p$ representation. A simple example is $(n,p) = (15,2)$, and $(n_1,n_2,n_3,n_4) = (8,4,2,1)$, which follows from the binary representation $15 = (1111)_2$.
\par
Given such a decomposition~\eqref{eq:n-partition}, we study the factorization in the form of~\eqref{eq:J-factor}, where the matrix $A \in \reals^{n \times n}$ has the so-called \textit{hierarchically banded} structure.
\begin{definition}[Hierarchically banded matrices] \label{def:hb-mat}
    Given $n \in \natint_{\geq 2}$ and a partition~\eqref{eq:n-partition}, a real symmetric $n \times n$ matrix $A$ is called \emph{hierarchically banded (HB)} if there exists a sequence of symmetric matrices $A^{(k)} \in \reals^{m_k \times m_k}$, $k \in [\tau]$, such that the following three conditions hold.
    \begin{itemize}
        \item $A^{(1)} = A$.

        \item $A^{(\tau)} \in \Diag^{n_\tau}$ is diagonal.

        \item For all $k \in [\tau-1]$, the matrix $A^{(k)}$ can be partitioned as
        \begin{equation} \label{eq:A-mat}
            A^{(k)} = \begin{bmatrix} 
                A^{(k)}_{11} & A^{(k)}_{12} \\
                \big(A^{(k)}_{12} \big)\tran & A^{(k)}_{22} 
            \end{bmatrix},
        \end{equation}
        where $A^{(k)}_{11} \in \Diag^{n_k}$, $A^{(k)}_{12} \in \reals^{n_k \times m_k}$ have nonzero entries only on the diagonals, and the last submatrix satisfies $A^{(k)}_{22} = A^{(k+1)}$. 
    \end{itemize}
    Such a sequence $\{A^{(k)}\}_{k=1}^{\tau}$ is called the \emph{hierarchically banded (HB) sequence} of $A$, and the set of $n \times n$ hierarchically banded matrices is denoted by $\HB^n$.
\end{definition}
The hierarchically banded structure is illustrated in \cref{fig:hb-mat}. The word ``hierarchically'' means that the matrix can be hierarchically partitioned, and this term is inspired by the notion of hierarchical matrices (or $\cH$-matrices) (see, \eg, \cite{borm2003introduction}). In addition, recall that a symmetric $n \times n$ banded matrix $A$ satisfies
\[
    A_{ij} = 0 \quad \text{if $j < i - r$ or $j > i + r$},
\]
where $r \in [n]$ is called the \textit{bandwidth} of $A$.
\begin{figure}[tb]
    \centering
    \begin{minipage}{0.2\textwidth}
        \[
            \begin{bmatrix} A^{(1)}_{11} & A^{(1)}_{12} \\ \big(A^{(1)}_{12} \big)\tran & A^{(1)}_{22} \end{bmatrix} = \;
        \]
    \end{minipage}%
    \begin{minipage}{0.55\textwidth}
    \centering
    \begin{tikzpicture}
        \node at (7.5,0) {$A^{(1)}_{22} = A^{(2)}$};
        \draw[gray!70] (0,0) -- (0,4) -- (4,4) -- (4,0) -- (0,0);
        \draw[gray!70] (2.5,0) -- (2.5,4);
        \draw[gray!70] (0,1.5) -- (4,1.5);
        \draw[thick] (0,4) -- (2.5,1.5);
        \draw[thick] (0,1.5) -- (1.5,0);
        \draw[thick] (2.5,4) -- (4,2.5);
        \filldraw[gray!40] (2.5,0) -- (4,0) -- (4,1.5) -- (2.5,1.5) -- (2.5,0);
        \draw[gray!40] (4,0) -- (6,0.5);
        \draw[gray!40] (4,1.5) -- (6,3.5);
        \draw[gray!70] (6,0.5) -- (9,0.5) -- (9,3.5) -- (6,3.5) -- (6,0.5);
        \draw[gray!70] (63/8,0.5) -- (63/8,3.5);
        \draw[gray!70] (6,13/8) -- (9,13/8);
        \draw[thick] (6,3.5) -- (63/8,13/8);
        \draw[thick] (6,13/8) -- (57/8,0.5);
        \draw[thick] (63/8,3.5) -- (9,19/8);
        \filldraw[gray!40] (63/8,0.5) -- (63/8,13/8) -- (9,13/8) -- (9,1/2) -- (63/8,1/2);
    \end{tikzpicture}
    \end{minipage}%
    \begin{minipage}{0.2\textwidth}
        \[
            \; = \begin{bmatrix} A^{(2)}_{11} & A^{(2)}_{12} \\ \big(A^{(2)}_{12} \big)\tran & A^{(2)}_{22} \end{bmatrix}
        \]
    \end{minipage}
    \caption{Illustration of a hierarchically banded matrix.}
    \label{fig:hb-mat}
\end{figure}
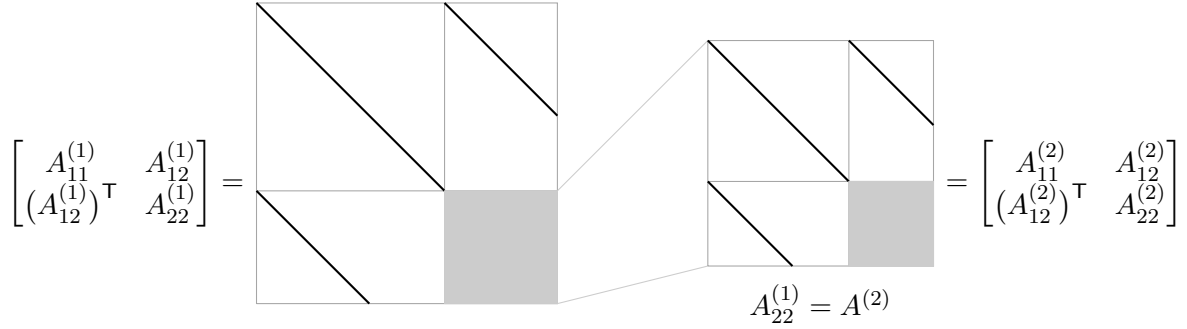
\par
Factorization~\eqref{eq:J-factor} with a hierarchically banded factor $A$ is called the \textit{hierarchically banded (HB) factorization} of $J$, and the matrix $A \in \HB^n$ is called the \textit{hierarchically banded (HB) factor} of~$J$. It turns out that the hierarchically banded factor $A$ is not unique. In this paper, we study the following two types of hierarchically banded factorization, characterized by additional properties of~$A$ or the HB sequence $A^{(k)}$ defined in~\eqref{eq:A-mat}.
\begin{itemize}
    \item \textit{Reduced hierarchically banded (RHB) factorization.}
    To further promote the sparsity of the $A$-factor, we impose an additional condition that only a few elements in the two bands of each~$A^{(k)}$ are nonzero:
    \[
        A^{(k)}_{12} [j,j] \neq 0, \ \ \text{if} \ j = 1, 1 + n_{k+1}, 1 + n_{k+1} + n_{k+2}, \ldots, 1 + \sum_{\ell = 1}^{\tau-k-1} n_{k+\ell},
    \]
    for all $k \in [\tau-1]$. In $A^{(1)}$, for example, this condition means that the largest cluster (the one of size $n_1$) communicates with exactly \textit{one} agent from each of the other clusters. Such a condition would further reduce the communication overhead, and the HB factor~$A$ designed for this purpose is called the (density) reduced HB factor, which is studied in \Cref{sec:rhb}.

    \item \textit{Doubly stochastic hierarchically banded (DSHB) factorization.} In this case, the factor $A$ is both hierarchically banded and \textit{doubly stochastic}, \ie, all the entries in~$A$ are nonnegative, $A\ones = \ones$, and $A\tran \ones = \ones$. This additional property of $A$ would be useful in decentralized optimization. The details are discussed in \Cref{sec:dshb}.
\end{itemize}
Moreover, the DSHB factorization inspires another sparse factorization~\eqref{eq:J-factor} of $J$, which is called the \textit{sequential doubly stochastic (SDS) factorization}. In this factorization, the SDS factor $A$ is doubly stochastic and can be written as the product of a sequence of symmetric, doubly stochastic matrices. Although the SDS factor is not hierarchically banded (nor symmetric), it is closely related to the DSHB factorization and finds its application in decentralized optimization. The details of the SDS factorization are presented in \Cref{sec:sds}.

%%%%%%%%%%%%%%%
%%  Section  %%
%%%%%%%%%%%%%%%
\section{Reduced hierarchically banded factorization} \label{sec:rhb}
As discussed in \Cref{sec:hb-mat}, the (density) reduced hierarchically banded (RHB) factorization further promotes sparsity in the HB factor $A$ by requiring
\begin{equation} \label{eq:rhb-off-diag}
    A^{(k)}_{12} [j,j] \neq 0, \ \ \text{if} \ j = 1, \ 1 + n_{k+1}, \ 1 + n_{k+1} + n_{k+2}, \ \ldots, \ 1 + \sum_{\ell = 1}^{\tau-k-1} n_{k+\ell},
\end{equation}
\ie, only a few nonzeros exist in the diagonal entries of $A^{(k)}_{12}$. In addition, it is also assumed that only one diagonal entry in each $A^{(k)}_{11}$ is not one:
\begin{equation} \label{eq:rhb-diag}
    A^{(k)}_{11} = \diag(\alpha_k, 1, \ldots, 1),
\end{equation}
for some $\alpha_k \in \reals$. (Other requirements on the diagonal submatrices $\{A^{(k)}_{11}\}$ can be applied, and they do not affect the idea of density reduction in the RHB factorization. So~\eqref{eq:rhb-diag} is chosen for simplicity.) The RHB factorization is illustrated in \Cref{sec:rhb-2x2} via the simple example where $\tau=2$, and \Cref{sec:rhb-algo} presents an algorithm for the RHB factorization in the general case.

%%%%%%%%%%%%%%%%%%
%%  Subsection  %%
%%%%%%%%%%%%%%%%%%
\subsection{A two-block example} \label{sec:rhb-2x2}
To illustrate the idea of the RHB factorization, we consider the simple case: $\tau = 2$. In this case, suppose that $n = n_1 + n_2$ with $(n_1, n_2) \in \natint_{\geq n_2} \times \natint_{\geq 1}$. Then, the HB factorization~\eqref{eq:J-factor} reduces to
\begin{equation} \label{eq:rhb-factor}
    J = \begin{bmatrix} J_1 & \\ & J_2 \end{bmatrix} 
    \begin{bmatrix} 
        A_{11} & A_{12} \\
        A_{12}\tran & A_{22}
    \end{bmatrix} 
    \begin{bmatrix} J_1 & \\ & J_2 \end{bmatrix} 
    = \begin{bmatrix} 
        J_1 A_{11} J_1 & J_1 A_{12} J_2 \\ 
        \big(J_1 A_{12} J_2\big)\tran & J_2 A_{22} J_2
    \end{bmatrix},
\end{equation}
where $J_1 = \tfrac{1}{n_1} \ones \ones\tran \in \reals^{n_1 \times n_1}$, $J_2 = \tfrac{1}{n_2} \ones \ones\tran \in \reals^{n_2 \times n_2}$, $A_{11} \in \Diag^{n_1}$, $A_{12} \in \reals^{n_1 \times n_2}$, and $A_{22} \in \Diag^{n_2}$. Expanding~\eqref{eq:rhb-factor} yields
\begin{equation} \label{eq:rhb-2x2}
    J_1 A_{11} J_1 = \frac{1}{n} \ones_{n_1} \ones_{n_1}\tran, \qquad
    J_1 A_{12} J_2 = \frac{1}{n} \ones_{n_1} \ones_{n_2}\tran, \qquad
    J_2 A_{22} J_2 = \frac{1}{n} \ones_{n_2} \ones_{n_2}\tran.
\end{equation} 
Recall that the condition~\eqref{eq:rhb-diag} requires $A_{11}$ to take the form $A_{11} = \diag(\alpha_1, 1, \ldots, 1)$ for some $\alpha_1 \in \reals$. Substituting it into $J_1 A_{11} J_1$ gives
\[
    J_1 A_{11} J_1 = \frac{1}{n_1^2} \ones_{n_1} \Big( \ones_{n_1}\tran \diag(\alpha_1, 1, \ldots, 1) \ones_{n_1} \Big) \ones_{n_1}\tran = \frac{\alpha_1 + n_1 - 1}{n_1^2} \ones_{n_1} \ones_{n_1}\tran.
\]
Then, combining it with the first condition in~\eqref{eq:rhb-2x2} yields
\[
    \alpha_1 = \frac{n_1^2}{n} - n_1 + 1.
\]
Similarly, one obtains that $A_{22} = \diag(\alpha_2, 1, \ldots, 1)$ with $\alpha_2 = \tfrac{n_2^2}{n} - n_2 + 1$. Finally, the condition~\eqref{eq:rhb-off-diag} implies that $A_{12}$ is nonzero only at the first element: $A_{12}[1,1] := \beta$, and then expanding the second block $J_1 A_{12} J_2$ with the condition~\eqref{eq:rhb-diag} gives
\[
    J_1 A_{12} J_2 = \frac{1}{n_1 n_2} \ones_{n_1} \Big( \ones_{n_1}\tran A_{12} \ones_{n_2} \Big) \ones_{n_2}\tran
    = \frac{\beta}{n_1 n_2} \ones_{n_1} \ones_{n_2}\tran,
\]
Combining it with the second condition in~\eqref{eq:rhb-2x2} gives
$A_{12}[1,1] = \beta = \frac{n_1 n_2}{n}$.
\par
In conclusion, when $n = n_1 + n_2$, the RHB factor $A$ of $J$ is given by 
\begin{equation} \label{eq:rhb-2x2-result}
    A = \left[ \begin{array}{cccccc|cccc}
        \alpha_1 & & & & & & \beta & & & \\
        & 1 & & & & & & 0 & & \\
        & & 1 & & &&  & & \ddots & \\
        & & & \ddots & & & & & & 0 \\
        & & & & \ddots & & & & & \\
        & & & & & 1 & & & & \\ \hline
        \beta & & & & & & \alpha_2 & & & \\
        & 0 & & & & & & 1 & & \\
        & & \ddots & & & & & & \ddots & \\
        & & & 0 & & & & & & 1
    \end{array} \right],
\end{equation}
where $\alpha_1 = \frac{n_1^2}{n} - n_1 + 1$, $\alpha_2 = \frac{n_2^2}{n} - n_2 + 1$, and $\beta = \frac{n_1 n_2}{n}$.

%%%%%%%%%%%%%%%%%%
%%  Subsection  %%
%%%%%%%%%%%%%%%%%%
\subsection{The RHB factorization algorithm} \label{sec:rhb-algo}
In this section, we extend the key idea in \Cref{sec:rhb-2x2} to handle the general case $n = \sum_{k=1}^\tau n_k$, where we denote $m_k := \sum_{i=k+1}^\tau n_i$ for $k \in [\tau-1]$ and assume that $n_k \geq m_k$ for all $k \in [\tau-1]$. Then, the construction of the RHB factorization of $J$ is summarized in \cref{algo:rhb}, which outputs the RHB factor $A$ that satisfies~\eqref{eq:J-factor}, \eqref{eq:rhb-off-diag}, and \eqref{eq:rhb-diag}.
\begin{algorithm}
\caption{Reduced hierarchically banded (RHB) factorization algorithm}
\label{algo:rhb}
\begin{algorithmic}[1]
    \State \textbf{Input:} $n \in \natint_{\geq 2}$, and the factors $\{n_k\}_{k=1}^\tau$ satisfying $n = \sum_{k=1}^\tau n_k$ and $n_k \geq m_k = \sum_{i=k+1}^\tau n_i$ for all $k \in [\tau-1]$. Denote $m_0 = n$.

    \State \textbf{Output:} The RHB factor $A$ of $J$, and the associated HB sequence $\{A^{(k)}\}_{k=1}^{\tau}$.

    \For{$k = 1, 2, \ldots, \tau-2$}
        \State Compute the $(1,1)$-block $A^{(k)}_{11} \in \Diag^{n_k}$ of $A^{(k)}$: \label{algo:rhb-step-1}
            \[
                A^{(k)}_{11} \leftarrow \diag \Big(\tfrac{n_k^2}{n} - n_k + 1, 1, \ldots, 1 \Big).
            \]

        \State Compute the $(1,2)$-block $A^{(k)}_{12} \in \reals^{n_k \times m_k}$: \label{algo:rhb-step-2}
            \[
                A^{(k)}_{12} [i,j] \leftarrow \begin{cases}
                    \tfrac{n_k n_{k+1}}{n} \quad &\text{if } i = j = 1 \\
                    \tfrac{n_k n_{k+\ell}}{n} \quad &\text{if } i = j = 1 + \sum\limits_{r=1}^{\ell} n_{k+r} \text{ for } \ell = 1, 2, \ldots, \tau-k-1 \\
                    0 &\text{otherwise.}
                \end{cases}
            \]

        \State Compute the $(2,2)$-block $A^{(k)}_{22} = A^{(k+1)}$ as the RHB factorization: \label{algo:rhb-step-3}
        \begin{equation} \label{eq:rhb-update-rule}
            \tfrac{1}{m_k} \ones_{m_k} \ones_{m_k}\tran = \Jbar_k A^{(k+1)} \Jbar_k,
        \end{equation}
        where $\Jbar_k := J_{k+1} \oplus \cdots \oplus J_\tau$, and the RHB factor $A^{(k+1)} = A^{(k)}_{22}$ is partitioned~as
        \begin{equation} \label{eq:rhb-update-3}
            A^{(k+1)} = \begin{bmatrix} 
                A^{(k+1)}_{11} & A^{(k+1)}_{12} \\ 
                \big(A^{(k+1)}_{12} \big)\tran & A^{(k+1)}_{22}
            \end{bmatrix}.
        \end{equation}
    \EndFor

    \State Set the RHB factor $A$: $A \leftarrow A^{(1)}$.
\end{algorithmic}
\end{algorithm}
\par
To verify the correctness of \cref{algo:rhb}, we start with the case $k=1$ and write out the equality $J = J_0 A J_0$ for the partitioned matrices:
\begin{equation} \label{eq:rhb-decomp}
    \frac{1}{n} \ones_n \ones_n\tran = \begin{bmatrix} J_1 & \\ & \Jbar_1 \end{bmatrix} 
    \begin{bmatrix} A^{(1)}_{11} & A^{(1)}_{12} \\ \big(A^{(1)}_{12} \big)\tran & A^{(1)}_{22} \end{bmatrix} 
    \begin{bmatrix} J_1 & \\ & \Jbar_1 \end{bmatrix},
\end{equation} 
where $\Jbar_1 := J_2 \oplus J_3 \oplus \cdots \oplus J_\tau \in \reals^{m_1 \times m_1}$. Expanding the above equation gives three conditions similar to~\eqref{eq:rhb-2x2}:
\begin{equation} \label{eq:rhb-cond}
    J_1 A^{(1)}_{11} J_1 = \frac{1}{n} \ones_{n_1} \ones_{n_1}\tran, \qquad
    J_1 A^{(1)}_{12} \Jbar_1 = \frac{1}{n} \ones_{n_1} \ones_{m_1}\tran, \qquad
    \Jbar_1 A^{(1)}_{22} \Jbar_1 = \frac{1}{n} \ones_{m_1} \ones_{m_1}\tran.
\end{equation}
It then follows from the condition~\eqref{eq:rhb-diag} that
\[
    J_1 A^{(1)} J_1 = \frac{1}{n_1^2} \ones_{n_1} \Big(\ones_{n_1}\tran \diag \big( \alpha_1,1,\ldots,1 \big) \ones_{n_1} \Big) \ones_{n_1}\tran = \frac{\alpha_1 + n_1 - 1}{n_1^2} \ones_{n_1} \ones_{n_1}\tran.
\]
Combining it with the first condition in~\eqref{eq:rhb-cond} yields $\alpha_1 = \tfrac{n_1^2}{n} - n_1 + 1$.
\par
We now consider the $(1,2)$-block $J_1 A^{(1)}_{12} \Jbar_1$. Recall from the condition~\eqref{eq:rhb-off-diag} that the matrix $A^{(1)}_{12} \in \reals^{n_1 \times m_1}$ can be partitioned as
\[
    A^{(1)}_{12} = \begin{bmatrix} 
        B^{(1)}_2 & & & & \\ 
        & B^{(1)}_3 & & & \\ 
        & & \ddots & & \\ 
        & & & B^{(1)}_{\tau-1} & \\ 
        & & & & B^{(1)}_\tau \\ 
        \zeros_2 & \zeros_3 & \cdots & \zeros_{\tau-1} & \zeros_\tau
    \end{bmatrix},
\]
where $B^{(1)}_j = \diag \big(\beta^{(1)}_j, 0, \ldots, 0\big) \in \Diag^{n_j}$, and $\zeros_j$ is the all-zeros matrix of size $(n_1 - m_1) \times n_j$, for $j=2,\ldots,\tau$.
In addition, we denote the diagonal entries of $B^{(1)}_j$ by the $n_j$-vector $b^{(1)}_j = (\beta^{(1)}_j, 0, \ldots, 0)$. Then, it holds that 
\[
    \ones_{n_1}\tran A^{(1)}_{12} = \begin{bmatrix} (b^{(1)}_2)\tran & (b^{(1)}_3)\tran & \cdots & (b^{(1)}_\tau)\tran \end{bmatrix} \in \reals^{1 \times m_1}.
\]
Then, it holds that
\begin{align*}
    J_1 A^{(1)}_{12} \Jbar &= \frac{1}{n_1} \ones_{n_1} (\ones_{n_1}\tran A^{(1)}_{12}) \Jbar_1 \\ 
    &= \frac{1}{n_1} \ones_{n_1} \begin{bmatrix} 
        \big(b^{(1)}_2 \big)\tran & \big(b^{(1)}_3 \big)\tran & \cdots & \big(b^{(1)}_\tau \big)\tran
    \end{bmatrix} \big(J_2 \oplus J_3 \oplus \cdots \oplus J_\tau \big) \\ 
    &= \frac{1}{n_1} \ones_{n_1} \begin{bmatrix}
       \frac{1}{n_2} \big(b^{(1)}_2 \big)\tran \ones_{n_2} \ones_{n_2}\tran &
       \frac{1}{n_3} \big(b^{(1)}_3 \big)\tran \ones_{n_3} \ones_{n_3}\tran & \cdots & 
       \frac{1}{n_\tau} \big(b^{(1)}_\tau \big)\tran \ones_{n_\tau} \ones_{n_\tau}\tran
   \end{bmatrix} \\ 
    &= \frac{1}{n_1} \ones_{n_1} \begin{bmatrix} \frac{\beta^{(1)}_2}{n_2} \ones_{n_2}\tran & \frac{\beta^{(1)}_3}{n_2} \ones_{n_3}\tran & \cdots & \frac{\beta^{(1)}_\tau}{n_\tau} \ones_{n_\tau}\tran \end{bmatrix} \\
    &= \begin{bmatrix}
        \frac{\beta^{(1)}_2}{n_1 n_2} \ones_{n_1} \ones\tran_{n_2} &
        \frac{\beta^{(1)}_3}{n_1 n_3} \ones_{n_1} \ones\tran_{n_3} & \cdots &
        \frac{\beta^{(1)}_\tau}{n_1 n_\tau} \ones_{n_1} \ones\tran_{n_\tau}
    \end{bmatrix} \in \reals^{n_1 \times m_1}.
\end{align*}
Hence, to satisfy the second condition in~\eqref{eq:rhb-cond}, we must have for all $j=2,\ldots,\tau$ that
\[
    \frac{\beta^{(1)}_j}{n_1 n_j} = \frac{1}{n} \qquad \Longleftrightarrow \qquad \beta^{(1)}_j = \frac{n_1 n_j}{n}.
\]
\par
Finally, we consider the last condition in~\eqref{eq:rhb-cond}. Denote $A^{(2)} := A^{(1)}_{22}$ and consider the partition~\eqref{eq:rhb-update-3}. Also notice that $\Jbar_1 = J_2 \oplus (J_3 \oplus \cdots \oplus J_\tau) := J_2 \oplus \Jbar_2$. Then, we write out the last condition~\eqref{eq:rhb-cond} in the partitioned form:
\[
    \frac{1}{n} \ones_{m_1} \ones_{m_1}\tran = \begin{bmatrix} J_2 & \\ & \Jbar_2 \end{bmatrix} 
    \begin{bmatrix} A^{(2)}_{11} & A^{(2)}_{12} \\ \big(A^{(2)}_{12} \big)\tran & A^{(2)}_{22} \end{bmatrix} 
    \begin{bmatrix} J_2 & \\ & \Jbar_2 \end{bmatrix},
\]
which takes the same form as~\eqref{eq:rhb-decomp}. We can then repeat the above process for $k=1,2,\ldots,\tau-2$. When \cref{algo:rhb} reaches iteration $k=\tau-2$, Line~\ref{algo:rhb-step-3} computes the RHB factorization of the matrix
\[
    A^{(\tau-1)} = \begin{bmatrix} 
        A^{(\tau-1)}_{11} & A^{(\tau-1)}_{12} \\ \big(A^{(\tau-1)}_{12} \big)\tran & A^{(\tau-1)}_{22}
    \end{bmatrix},
\]
which is the two-block case studied in \Cref{sec:rhb-2x2}. Thus, the RHB factor of $A^{(\tau-1)}$ is in the form of~\eqref{eq:rhb-2x2-result} with
\[
    \alpha_1 = \frac{n_{\tau-1}^2}{n} - n_{\tau-1} + 1, \qquad \alpha_2 = \frac{n_\tau^2}{n} - n_\tau + 1, \qquad \beta = \frac{n_{\tau-1} n_\tau}{n}.
\]
\par
From the above discussion, we obtain the following result.
\begin{theorem} \label{thm:rhb}
    The $n \times n$ matrix $A_\mathrm{RHB}$ generated by \cref{algo:rhb} is hierarchically banded and satisfies $J=J_0 A_\mathrm{RHB} J_0$ as well as conditions \eqref{eq:rhb-off-diag}--\eqref{eq:rhb-diag}. In addition, the total number of nonzeros is $\nnz(A_\mathrm{RHB}) = n + \tau (\tau-1)$, and the largest node degree is $d_\mathrm{max} (A_\mathrm{RHB}) = \tau$.
\end{theorem}
\begin{proof}
    (The subscript in $A_\mathrm{RHB}$ is omitted in the proof for readability.) The hierarchically banded structure of $A$ follows from the recursive nature of \cref{algo:rhb}, and in particular, the recursive partition of~$A^{(k)}$ in Line~\ref{algo:rhb-step-3} of \cref{algo:rhb}. Similarly, $A$ satisfies the conditions~\eqref{eq:rhb-off-diag}--\eqref{eq:rhb-diag} due to the assignment of values in $A^{(k)}_{11}$ and $A^{(k)}_{12}$ in Lines~\ref{algo:rhb-step-1} and \ref{algo:rhb-step-2}. Next, the factorization $J=J_0 A J_0 = J_0 A^{(1)} J_0$ holds by recursively applying~\eqref{eq:rhb-update-rule} for $k=\tau-1, \tau-2, \ldots, 1$. Finally, one has for all $k \in [\tau-1]$ that $\nnz(A^{(k)}_{11}) = n_k$, $\nnz(A^{(k)}_{12}) = \tau-k$, and $\nnz(A^{(\tau)}) = n_\tau$. So, the total number of nonzeros is 
    \[
        \nnz(A) = \sum_{k=1}^{\tau-1} \big(n_k + 2(\tau-k) \big) + n_\tau = n + \tau (\tau-1).
    \]
    The row with the most nonzeros is row $n-n_\tau - n_{\tau - 1}+1$, where $A_{n-n_\tau - n_{\tau - 1}+1,j} \neq 0$ if $j = 1, 1+n_1, 1+n_1+n_2, \ldots, 1+ \sum_{k=1}^{\tau-1} n_k$, and thus $d_\mathrm{max} (A) = \tau$.
\end{proof}

%%%%%%%%%%%%%%%
%%  Section  %%
%%%%%%%%%%%%%%%
\section{Doubly stochastic hierarchically banded factorization} \label{sec:dshb}
Another useful type of hierarchically banded factorization, especially in decentralized optimization (see~\Cref{sec:decentr-opt} for details), requires the HB factor $A$ to be \textit{doubly stochastic}, \ie, all the entries are nonnegative and $A \ones = \ones$ (and $A\tran \ones = \ones$, which is guaranteed by the symmetry of $A$). Yet in this case, the HB sequence $\{A^{(k)}\}$ is \textit{not} doubly stochastic. Instead, we show that each matrix in the \textit{scaled} HB sequence $\{\Atilde^{(k)}\}_{k=1}^\tau$ remains doubly stochastic, where
\begin{equation} \label{eq:dshb-Atilde}
    \Atilde^{(k)} := \frac{n}{m_{k-1}} A^{(k)} \in \HB^{m_k}, \quad k \in [\tau].
\end{equation}
Again, the doubly stochastic hierarchically banded (DSHB) factorization is illustrated in \Cref{sec:dshb-2x2} via the simple example where $\tau=2$, and \Cref{sec:dshb-algo} presents an algorithm for DSHB factorization in the general case.

%%%%%%%%%%%%%%%%%%
%%  Subsection  %%
%%%%%%%%%%%%%%%%%%
\subsection{A two-block example} \label{sec:dshb-2x2}
Similar to \Cref{sec:rhb-2x2}, we start with the simple case where $\tau = 2$, and assume $n = n_1 + n_2$ with $(n_1, n_2) \in \natint_{\geq n_2} \times \natint_{\geq 1}$. Then, the HB factorization takes the form of~\eqref{eq:rhb-factor}, which can be partitioned as in~\eqref{eq:rhb-2x2}. Recall that in \Cref{sec:rhb-2x2} we require both submatrices $A_{11}$ and~$A_{12}$ to have only one nonzero entry. In the context of decentralized optimization, a larger cluster will communicate with exactly one agent from each of the smaller clusters. This section considers a different setting where all agents in subgroup 2 (recall $n_2 \leq n_1$) can communicate across subgroups. In particular, the submatrix $A_{12} \in \reals^{n_1 \times n_2}$ takes the following form:
\[
    A_{12} = \begin{bmatrix} \diag(\beta \ones_{n_2}) \\ 0 \end{bmatrix}.
\]
Substituting into~$J_1 A_{12} J_2$ gives
\[
    J_1 A_{12} J_2 = \frac{1}{n_1 n_2} \ones_{n_1} \Big(\ones_{n_1}\tran A_{12} \ones_{n_2} \Big) \ones_{n_2}\tran = \frac{\beta}{n_1} \ones_{n_1} \ones_{n_2}\tran.
\]
Then, the second condition in~\eqref{eq:rhb-2x2} implies that
\[
    \beta = \frac{n_1}{n} = \frac{n_1}{n_1 + n_2}.
\]
With the sub-block $A_{12}$ settled, the doubly stochastic property of $A$ implies that the sub-blocks $A_{11} \in \Diag^{n_1}$ and $A_{22} \in \Diag^{n_2}$ are diagonal matrices satisfying
\[
    A_{11} = \diag(\underbrace{1-\beta, \ldots, 1-\beta}_{n_2}, \underbrace{1, \ldots, 1}_{n_1-n_2}), \qquad 
    A_{22} = \diag\big( (1-\beta) \ones_{n_2} \big).
\]
Finally, we confirm that this choice of $A_{11}$ and $A_{22}$ also satisfies the first and third conditions in~\eqref{eq:rhb-2x2}:
\begin{align*}
    \tfrac{1}{n_1^2} \ones_{n_1} \ones_{n_1}\tran A_{11} \ones_{n_1} \ones_{n_1}\tran &= \tfrac{\ones_{n_1}\tran A_{11} \ones_{n_1}}{n_1^2} \ones_{n_1} \ones_{n_1}\tran = \tfrac{(1-\beta) n_2 + (n_1 - n_2)}{n_1^2} \ones_{n_1} \ones_{n_1}\tran = \tfrac{1}{n} \ones_{n_1} \ones_{n_1}\tran, \\
    \tfrac{1}{n_2^2} \ones_{n_2} \ones_{n_2}\tran A_{22} \ones_{n_2} \ones_{n_2}\tran &= \tfrac{\ones_{n_2}\tran A_{22} \ones_{n_2}}{n_2^2} \ones_{n_2} \ones_{n_2}\tran = \tfrac{(1-\beta) n_2}{n_2^2} \ones_{n_2} \ones_{n_2}\tran = \tfrac{1}{n} \ones_{n_2} \ones_{n_2}\tran.
\end{align*}
\par
In conclusion, when $n = n_1 + n_2$, the doubly stochastic HB factor $A$ of $J$ is 
\begin{equation} \label{eq:dshb-2x2-result}
    A = \left[ \begin{array}{cc|c}
        \tfrac{n_2}{n} I_{n_2} & 0 & \tfrac{n_1}{n} I_{n_2} \\ 
        0 & I_{n_1 - n_2} & 0 \\ \hline
        \tfrac{n_1}{n} I_{n_2} & 0 & \tfrac{n_2}{n} I_{n_2}
    \end{array} \right].
\end{equation}

%%%%%%%%%%%%%%%%%%
%%  Subsection  %%
%%%%%%%%%%%%%%%%%%
\subsection{The DSHB factorization algorithm} \label{sec:dshb-algo}
We extend the key idea in \Cref{sec:rhb-2x2} to handle the general case where $n = \sum_{k=1}^\tau n_k$. The construction of the DSHB factor $A$, as well as the associated (scaled) HB sequence $\{A^{(k)}\}$ ($\{\Atilde^{(k)}\}$ in~\eqref{eq:dshb-Atilde}), is summarized in \cref{algo:dshb}.
\begin{algorithm}[tb]
\caption{Doubly stochastic hierarchically banded (DSHB) factorization algorithm}
\label{algo:dshb}
\begin{algorithmic}[1]
    \State \textbf{Input:} $n \in \natint_{\geq 2}$, and the factors $\{n_k\}_{k=1}^\tau$ satisfying $n = \sum_{k=1}^\tau n_k$ and $n_k \geq m_k = \sum_{i=k+1}^\tau n_i$ for all $k \in [\tau-1]$.

    \State \textbf{Output:} The doubly stochastic HB factor $A$ of $J$, and the associated HB sequence $\{A^{(k)}\}_{k=1}^\tau$.

    \State Set $m_{-1} \leftarrow n$ and $m_0 \leftarrow n$.

    \For{$k=1,2,\ldots,\tau-2$}
        \State Compute the $(1,1)$-block $\Atilde^{(k)}_{11} \in \Diag^{n_k}$ of $\Atilde^{(k)}$:
            \begin{equation} \label{eq:dshb-step-1}
                \Atilde^{(k)}_{11} \leftarrow \diag \big(\underbrace{\tfrac{m_k}{m_{k-1}}, \ldots, \tfrac{m_k}{m_{k-1}}}_{m_k}, \underbrace{1,\ldots,1}_{n_k - m_k} \big).
            \end{equation}

        \State Compute the $(1,2)$-block $\Atilde^{(k)}_{12} \in \reals^{n_k \times m_k}$:
            \begin{equation} \label{eq:dshb-step-2}
                \Atilde^{(k)}_{12} [i,j] \leftarrow \begin{cases}
                    \tfrac{n_k}{m_{k-1}} \quad &\text{if } i = j = 1,2,\ldots,m_k \\ 
                    0 &\text{otherwise.}
                \end{cases}
            \end{equation} 

        \State Compute the $(2,2)$-block $\Atilde^{(k)}_{22}$ from the DSHB factorization:
            \begin{equation} \label{eq:dshb-update-rule}
                \frac{1}{m_k} \ones_{m_k} \ones_{m_k}\tran = \Jbar_k \Atilde^{(k+1)} \Jbar_k,
            \end{equation}
            where $\Jbar_k := J_{k+1} \oplus \cdots \oplus J_\tau$, and the DSHB factor $\Atilde^{(k+1)} \leftarrow \tfrac{m_{k-1}}{m_k} \Atilde^{(k)}_{22}$ is partitioned as \label{algo:dshb-step-3}
            \[
                \Atilde^{(k+1)} = \begin{bmatrix} 
                    \Atilde^{(k+1)}_{11} & \Atilde^{(k+1)}_{12} \\ 
                    \big(\Atilde^{(k+1)}_{12} \big)\tran & \Atilde^{(k+1)}_{22}
                \end{bmatrix}.
            \]
    \EndFor

    \State Set the DSHB factor $A \leftarrow A^{(1)} \equiv \Atilde^{(1)}$ and the associated HB sequence $A^{(k)} \leftarrow \tfrac{m_{k-1}}{n} \Atilde^{(k)}$, for all $k \in [\tau]$.
\end{algorithmic}
\end{algorithm}
\par
To verify the correctness of \cref{algo:dshb}, we start with the iteration $k=1$ and write out the equality $J = J_0 A J_0$ for the partitioned matrices:
\[
    \frac{1}{n} \ones_n \ones_n\tran = 
    \begin{bmatrix} J_1 & \\ & \Jbar_1 \end{bmatrix} 
    \begin{bmatrix} A^{(1)}_{11} & A^{(1)}_{12} \\ \big(A^{(1)}_{12} \big)\tran & A^{(1)}_{22} \end{bmatrix} 
    \begin{bmatrix} J_1 & \\ & \Jbar_1 \end{bmatrix},
\]
where recall $\Jbar_1 := J_2 \oplus J_3 \oplus \cdots \oplus J_\tau \in \reals^{m_1 \times m_1}$ and $A^{(1)} = \Atilde^{(1)}$. Expanding the above equation gives three conditions similar to~\eqref{eq:rhb-2x2}:
\begin{equation} \label{eq:dshb-cond}
    J_1 A^{(1)}_{11} J_1 = \frac{1}{n} \ones_{n_1} \ones_{n_1}\tran, \qquad
    J_1 A^{(1)}_{12} \Jbar_1 = \frac{1}{n} \ones_{n_1} \ones_{m_1}\tran, \qquad
    \Jbar_1 A^{(1)}_{22} \Jbar_1 = \frac{1}{n} \ones_{m_1} \ones_{m_1}\tran.
\end{equation} 
For the $(1,2)$-block $A^{(1)}_{12}$, we follow the convention in \Cref{sec:dshb-2x2} and assume that it has the structure
\[
    A^{(1)}_{12} = \begin{bmatrix} \diag \big(\beta^{(1)} \ones_{m_1} \big) \\ 0 \end{bmatrix}.
\]
Substituting into~$J_1 A^{(1)}_{12} \Jbar_1$ gives
\[
    J_1 A^{(1)}_{12} \Jbar_1 = \frac{1}{n_1} \ones_{n_1} \big(\ones_{n_1}\tran A^{(1)}_{12} \big) \Jbar_1 
    = \frac{\beta^{(1)}}{n_1} \ones_{n_1} \ones_{m_1}\tran \Jbar_1 = \frac{\beta^{(1)}}{n_1} \ones_{n_1} \ones_{m_1}\tran.
\]
Combining it with the second condition in~\eqref{eq:dshb-cond} yields $\beta^{(1)} = \tfrac{n_1}{n}$. Then, the doubly stochastic property of $A$ implies that
\[
    A^{(1)}_{11} = \diag \big(\underbrace{\tfrac{m_1}{n}, \ldots, \tfrac{m_1}{n}}_{m_1}, \underbrace{1,\ldots,1}_{n_1-m_1} \big), \qquad A^{(1)}_{22} \ones_{m_1} = (1-\beta^{(1)}) \ones_{m_1} = \tfrac{m_1}{n} \ones_{m_1}.
\]
The second equation above is equivalent to the doubly stochastic property of the \textit{scaled} matrix
\[
    \Atilde^{(2)} \ones_{m_1} = \ones_{m_1}, \quad \text{where } \Atilde^{(2)} := \frac{n}{m_1} A^{(1)}_{22} \in \reals^{m_1 \times m_1}.
\]
Similarly, the third condition in~\eqref{eq:dshb-cond} can be written in terms of $\Atilde^{(2)}$ as
\begin{equation} \label{eq:dshb-tilde}
    \Jbar_1 \Atilde^{(2)} \Jbar_1 = \frac{1}{m_1} \ones_{m_1} \ones_{m_1}\tran.
\end{equation}
\par
Therefore, to find a doubly stochastic, hierarchically banded matrix $\Atilde^{(2)}$ that satisfies~\eqref{eq:dshb-tilde}, we need to construct the DSHB factorization of $\tfrac{1}{m_1} \ones_{m_1} \ones_{m_1}\tran$, which requires recursive execution of the above process for $k=1,2,\ldots,\tau-2$.
\par
When \cref{algo:dshb} reaches iteration $k=\tau-2$, Line~\ref{algo:dshb-step-3} computes the DSHB factor
\[
    \Atilde^{(\tau-1)} = \begin{bmatrix} 
        \Atilde^{(\tau-1)}_{11} & \Atilde^{(\tau-1)}_{12} \\ 
        \big(\Atilde^{(\tau-1)}_{12} \big)\tran & \Atilde^{(\tau-1)}_{22} 
    \end{bmatrix},
\]
which is the two-block case studied in \Cref{sec:dshb-2x2}. Thus, the DSHB factor of $\tfrac{1}{m_{\tau-2}} \ones_{m_{\tau-2}} \ones_{m_{\tau-2}}\tran$ is in the form of~\eqref{eq:dshb-2x2-result}:
\[
    \Atilde^{(\tau-1)} = \left[ \begin{array}{cc|c}
        \alpha I_{n_\tau} & 0 & \beta I_{n_\tau} \\ 
        0 & I_{n_{\tau-1} - n_\tau} & 0 \\ \hline
        \beta I_{n_\tau} & 0 & \alpha I_{n_\tau}
    \end{array} \right], \;\; \text{where} \
    \alpha = \frac{n_\tau}{n_{\tau-1} + n_\tau} \ \text{and} \
    \beta = \frac{n_{\tau-1}}{n_{\tau-1} + n_\tau}.
\]
\par
From the above discussion, we obtain the following result.
\begin{theorem} \label{thm:dshb}
    The $n \times n$ matrix $A_\mathrm{DSHB}$ generated by \cref{algo:dshb} is doubly stochastic, hierarchically banded, and satisfies $J = J_0 A_\mathrm{DSHB} J_0$. Each matrix in the scaled HB sequence $\{\Atilde^{(k)}_\mathrm{DSHB}\}_{k=1}^\tau$ generated by \cref{algo:dshb} is doubly stochastic. In addition, it holds that
    \[
        \nnz(A_\mathrm{DSHB}) = \sum_{k=1}^\tau k n_k, \qquad d_\mathrm{max} (A_\mathrm{DSHB}) = \tau.
    \]
\end{theorem}
\begin{proof}
    (The subscript in $A_\mathrm{DSHB}$ (and $\Atilde^{(k)}_\mathrm{DSHB}$) is omitted in the proof for readability.) The doubly stochastic property of $A$ and $\{\Atilde^{(k)}\}$ follows from the assignments~\eqref{eq:dshb-step-1}--\eqref{eq:dshb-step-2} and condition~\eqref{eq:dshb-update-rule}. The hierarchically banded structure of $A$ follows from the recursive nature of \cref{algo:dshb}, and in particular, the recursive partition of~$\Atilde^{(k)}$ in Line~\ref{algo:dshb-step-3}. Next, the factorization $J=J_0 A J_0$ holds by recursively applying~\eqref{eq:dshb-update-rule} for $k=\tau-1, \tau-2, \ldots, 1$. Finally, the number of nonzeros and the largest node degree can be calculated using the same approach as for the RHB factor in \cref{thm:rhb}.
\end{proof}

%%%%%%%%%%%%%%%
%%  Section  %%
%%%%%%%%%%%%%%%
\section{Sequential doubly stochastic factorization} \label{sec:sds}
The DSHB factorization inspires another type of factorization for~$J$, in which the factor $A \in \reals^{n \times n}$ in $J = J_0 A J_0$ is no longer symmetric (nor hierarchically banded) but remains doubly stochastic. Since the asymmetric, doubly stochastic matrix $A$ can be written as the product of a sequence of doubly stochastic matrices, such a factorization is called the \textit{sequential doubly stochastic (SDS) factorization} of $J$.
\begin{theorem}[Sequential doubly stochastic (SDS) factorizations of $J$] \label{thm:sds}
    Let $A \in \HB^n$ be the DSHB factor of $J$ and $\{\Atilde^{(k)}\}_{k=1}^\tau$ the associated scaled HB sequence, constructed via \cref{algo:dshb}. For all $k \in [\tau-1]$, define
    \begin{equation} \label{eq:sds-T}
        T^{(k)} := \begin{bmatrix} 
            \Atilde^{(k)}_{11} & \Atilde^{(k)}_{12} \\ \big(\Atilde^{(k)}_{12} \big)\tran & \tfrac{m_k}{m_{k-1}} I_{m_k}
        \end{bmatrix} \in \reals^{m_{k-1} \times m_{k-1}},
    \end{equation} 
    with the convention $m_0 := n$, and $T^{(\tau)} := \Atilde^{(\tau)} \equiv I_{n_\tau}$. The augmented matrices $\{\That^{(k)}\}_{k=1}^\tau \subset \reals^{n \times n}$ are defined as
    \[
        \That^{(k)} = I_{n_1} \oplus I_{n_2} \oplus \cdots \oplus I_{n_{k-1}} \oplus T^{(k)}, \quad k \in [\tau].
    \]
    Then, the matrices $\{T^{(k)}\}_{k=1}^\tau$ (and $\{\That^{(k)}\}_{k=1}^\tau$) are all symmetric, doubly stochastic, and the matrix $J = \tfrac{1}{n} \ones_n \ones_n\tran$ can be factored as
    \begin{equation} \label{eq:sds}
        J = J_0 A_\mathrm{L} J_0 = J_0 A_\mathrm{R} J_0,
    \end{equation} 
    where
    \begin{align*}
        A_\mathrm{L} &:= \That^{(1)} \That^{(2)} \cdots \That^{(\tau)} \\
        &\; = T^{(1)} \cdot (I_{n_1} \oplus (T^{(2)} \cdot (I_{n_2} \oplus \cdots (T^{(\tau-1)} \cdot (I_{n_\tau} \oplus T^{(\tau)}))))), \nonumber \\
        A_\mathrm{R} &:= \That^{(\tau)} \That^{(\tau-1)} \cdots \That^{(1)} \\
        &\; = (I_{n_1} \oplus \cdots \oplus (I_{n_{\tau-1}} \oplus (I_{n_\tau} \oplus T^{(\tau)}) \cdot T^{(\tau-1)}) \cdot T^{(\tau-2)}) \cdots T^{(1)}. \nonumber
    \end{align*}
    In addition, both factors $A_\mathrm{L}$ and $A_\mathrm{R}$ are doubly stochastic.
\end{theorem}
By definition, the matrices $\{T^{(k)}\}$ have nonzero entries only in three subdiagonals. The first factorization in~\eqref{eq:sds} is called the \emph{left SDS factorization} of $J$, and the second is called the \emph{right SDS factorization}.
\begin{proof}
    Define $\Jbar_0 := J_0$, $m_0 := n$, and
    \begin{align}
        V^{(k)} &:= T^{(k)} \cdot (I_{n_k} \oplus (T^{(k+1)} \cdot (I_{n_{k+1}} \oplus \cdots (T^{(\tau-1)} \cdot (I_{n_\tau} \oplus T^{(\tau)}))))) \nonumber \\
        &\; = T^{(k)} \cdot (I_{n_k} \oplus V^{(k+1)}) \in \reals^{m_{k-1} \times m_{k-1}}, \label{eq:sds-V}
    \end{align}
    for all $k \in [\tau-1]$, and $V^{(\tau)} := T^{(\tau)} \equiv I_{n_\tau}$. By definition, each matrix $V^{(k)}$ is doubly stochastic, because $T^{(k)}$ is doubly stochastic.
    \par
    First, we apply mathematical induction to prove that for $k = \tau-1, \dots, 1$,
    \begin{equation} \label{eq:sds-prf}
        \Jbar_{k-1} V^{(k)} \Jbar_{k-1} 
        = \tfrac{1}{m_{k-1}} \ones_{m_{k-1}} \ones_{m_{k-1}}\tran,
    \end{equation} 
    The base case $k \leftarrow \tau - 1$ holds because
    \begin{subequations}
    \begin{align}
        \Jbar_{\tau-2} V^{(\tau-1)} \Jbar_{\tau-2} &= 
            \begin{bmatrix} J_{\tau-1} & \\ & J_\tau \end{bmatrix} 
            T^{(\tau-1)} (I_{n_\tau} \oplus T^{(\tau)})
            \begin{bmatrix} J_{\tau-1} & \\ & J_\tau \end{bmatrix} \label{eq:sds-prf-base-1} \\ 
        &= \begin{bmatrix} J_{\tau-1} & \\ & J_\tau \end{bmatrix} T^{(\tau-1)} \begin{bmatrix} J_{\tau-1} & \\ & J_\tau \end{bmatrix} \label{eq:sds-prf-base-2} \\
        &= \begin{bmatrix} J_{\tau-1} & \\ & J_\tau \end{bmatrix} 
            \begin{bmatrix} 
                \Atilde^{(\tau-1)}_{11} & \Atilde^{(\tau-1)}_{12} \\
                \big(\Atilde^{(\tau-1)}_{12} \big)\tran & \tfrac{m_{\tau-1}}{m_{\tau-2}} I
            \end{bmatrix} \begin{bmatrix} J_{\tau-1} & \\ & J_\tau \end{bmatrix} \label{eq:sds-prf-base-3} \\ 
        &= \begin{bmatrix} 
                J_{\tau-1} \Atilde^{(\tau-1)}_{11} J_{\tau-1} & J_{\tau-1} \Atilde^{(\tau-1)}_{12} J_{\tau} \\ 
                \big(J_{\tau-1} \Atilde^{(\tau-1)}_{12} J_{\tau} \big)\tran & \tfrac{m_{\tau-1}}{m_{\tau-2}} J_\tau^2
            \end{bmatrix} \nonumber \\ 
        &= \tfrac{1}{m_{\tau-2}} \ones_{m_{\tau-2}} \ones_{m_{\tau-2}}\tran. \label{eq:sds-prf-base-4} 
    \end{align}
    \end{subequations}
    The first equation~\eqref{eq:sds-prf-base-1} uses $\Jbar_{\tau-2} = J_{\tau-1} \oplus \Jbar_{\tau-1} = J_{\tau-1} \oplus J_\tau$. Then,~\eqref{eq:sds-prf-base-2} and~\eqref{eq:sds-prf-base-3} use the definition $T^{(\tau)} = I_{n_\tau}$ and~\eqref{eq:sds-T}. Finally,~\eqref{eq:sds-prf-base-4} follows from the updates~\eqref{eq:dshb-step-1} and~\eqref{eq:dshb-step-2} in \cref{algo:dshb}, as well as the fact $J_\tau^2 = J_\tau$.
    \par
    Next, suppose the identity~\eqref{eq:sds-prf} holds for $k \in [\tau-1]$, and we establish the same identity with $k \leftarrow k-1$:
    \begin{subequations}
    \begin{align}
        \MoveEqLeft[0.2] {\Jbar_{k-2} T^{(k-1)} (I_{n_k} \oplus V^{(k)}) \Jbar_{k-2}} \nonumber \\ 
        &= \begin{bmatrix} J_{k-1} & \\ & \Jbar_{k-1} \end{bmatrix} \begin{bmatrix}
                \Atilde^{(k-1)}_{11} & \Atilde^{(k-1)}_{12} \\ 
                \big(\Atilde^{(k-1)} \big)\tran & \tfrac{m_{k-1}}{m_{k-2}} I
            \end{bmatrix} 
            \begin{bmatrix} I_{n_{k-1}} & \\ & V^{(k)} \end{bmatrix}
            \begin{bmatrix} J_{k-1} & \\ & \Jbar_{k-1} \end{bmatrix} \label{eq:sds-prf-induct-1} \\ 
        &= \begin{bmatrix} 
                J_{k-1} A^{(k-1)}_{11} J_{k-1} & J_{k-1} A^{(k-1)}_{12} V^{(k)} \Jbar_{k-1} \\ 
                \big(J_{k-1} A^{(k-1)}_{12} \Jbar_{k-1} \big)\tran & \tfrac{m_{k-1}}{m_{k-2}} \Jbar_{k-1} V^{(k)} \Jbar_{k-1}
            \end{bmatrix} \label{eq:sds-prf-induct-2} \\
        &= \begin{bmatrix} 
                \tfrac{1}{m_{k-2}} \ones_{n_{k-1}} \ones_{n_{k-1}}\tran & \tfrac{1}{m_{k-2}} \ones_{n_{k-1}} \ones_{m_{k-1}}\tran \\ 
                \tfrac{1}{m_{k-2}} \ones_{m_{k-1}} \ones_{n_{k-1}}\tran & \tfrac{1}{m_{k-2}} \ones_{m_{k-1}} \ones_{m_{k-1}}\tran
            \end{bmatrix} \label{eq:sds-prf-induct-3} \\ 
        &= \tfrac{1}{m_{k-2}} \ones_{m_{k-2}} \ones_{m_{k-2}}\tran. \nonumber
    \end{align}
    \end{subequations}
    In~\eqref{eq:sds-prf-induct-1}, we use $\Jbar_{k-2} = J_{k-1} \oplus \Jbar_{k-1}$, the definition of $T^{(k)}$ in~\eqref{eq:sds-T}, and the definition of direct sum. After multiplying out all the matrices in~\eqref{eq:sds-prf-induct-2}, the third equation~\eqref{eq:sds-prf-induct-3} follows from the definition of $A^{(k-1)}_{11}$ and $A^{(k-1)}_{12}$ (see~\eqref{eq:dshb-step-1}--\eqref{eq:dshb-step-2}), the definition of~$V^{(k)}$ in~\eqref{eq:sds-V}, and the assumption that identity~\eqref{eq:sds-prf} holds for $k \in [\tau-1]$. In particular, the $(1,2)$-block of~\eqref{eq:sds-prf-induct-2} is further simplified as follows:
    \begin{subequations}
    \begin{align}
        \MoveEqLeft[0.2] {J_{k-1} A^{(k-1)}_{12} V^{(k)} \Jbar_{k-1}} \nonumber \\ 
        &= \tfrac{1}{n_{k-1}} \ones_{n_{k-1}} \left( \ones_{n_{k-1}}\tran \begin{bmatrix} \diag \Big(\tfrac{n_{k-1}}{m_{k-2}} \ones_{m_{k-1}} \Big) \\ 0 \end{bmatrix} \right) V^{(k)} \Jbar_{k-1} \label{eq:sds-prf-induct-2a} \\ 
        &= \tfrac{1}{m_{k-2}} \ones_{n_{k-1}} \big(\ones_{m_{k-1}}\tran V^{(k)} \big) \Jbar_{k-1} \label{eq:sds-prf-induct-2b} \\ 
        &= \tfrac{1}{m_{k-2}} \ones_{n_{k-1}} \ones_{m_{k-1}}\tran \Jbar_{k-1} \label{eq:sds-prf-induct-2c} \\ 
        &= \tfrac{1}{m_{k-2}} \ones_{n_{k-1}} \ones_{m_{k-1}}\tran. \label{eq:sds-prf-induct-2d}
    \end{align}
    \end{subequations}
    In~\eqref{eq:sds-prf-induct-2a}, we use the definition of~$\Atilde^{(k-1)}_{12}$ in~\eqref{eq:dshb-step-2}, and~\eqref{eq:sds-prf-induct-2b} writes out $\ones_{n_{k-1}}\tran \Atilde^{(k-1)}_{12} = \frac{n_{k-1}}{m_{k-2}}\ones_{m_{k-1}}\tran$. Then, \eqref{eq:sds-prf-induct-2c} and~\eqref{eq:sds-prf-induct-2d} use the doubly stochastic property of $V^{(k)}$ and~$\Jbar_{k-1}$, respectively.
    \par
    Therefore, the induction hypothesis is proved, and~\eqref{eq:sds-prf} holds for all $k \in [\tau-1]$. In particular,~\eqref{eq:sds-prf} with $k \leftarrow 1$ gives the left SDS factorization:
    \[
        J_0 A_\mathrm{L} J_0 = \Jbar_0 V^{(1)} \Jbar_0 = \tfrac{1}{m_0} \ones_{m_0} \ones_{m_0}\tran = J.
    \]
    The first equation uses the convention $\Jbar_0 = J_0$ and the relation $A_\mathrm{L} = T^{(1)} (I_{n_1} \oplus V^{(2)}) = V^{(1)}$. The second equation applies~\eqref{eq:sds-prf} with $k=1$, and the last one follows from the convention $m_0 = n$. Then, the right SDS factorization follows directly from the fact that $A_\mathrm{R} = A_\mathrm{L}\tran$ and thus $J = (J_0 A_\mathrm{L} J_0)\tran = J_0 A_\mathrm{L}\tran J_0 = J_0 A_\mathrm{R} J_0$. Finally, the doubly stochastic property of $A_\mathrm{L}$ (and $A_\mathrm{R}$) follows from that of $\{T^{(k)}\}$ and the fact that the product of doubly stochastic matrices is still doubly stochastic.
\end{proof}
\par
In the context of decentralized optimization, if communication is modeled by the $T$-factors, then at each round of communication, each agent only needs to communicate with at most one neighbor (as $d_\mathrm{max} (T^{(k)}) = 2$ for all $k \in [\tau-1]$). Such a property is called ``one-peer'' in decentralized optimization and holds for one-peer hyper-cubes \cite{shi2016finite} and one-peer exponential graphs \cite{ying2021exponential}.
\par
We also note that the matrices $\{T^{(k)}\}$ represent the base-$(p+1)$ graphs introduced in \cite{takezawa2023exponential}. Yet, the original work \cite{takezawa2023exponential} fails to provide an explicit matrix representation for the base-$(p+1)$ graphs and does not prove that the weight matrices of their proposed base-$(p+1)$ graphs can be used to factorize the $J$ matrix. Moreover, as explained in \Cref{sec:hb-mat}, the construction of all the matrices ($A_\mathrm{L}$, $A_\mathrm{R}$, $\{T^{(k)}\}$, and $\{\Atilde^{(k)}\}$) does not necessarily rely on the base-$p$ representation of the integer $n \in \natint_{\geq 2}$, and only needs a decomposition $n = \sum_{k=1}^\tau n_k$ with $n_k \geq m_k = \sum_{i=k+1}^\tau n_i$ for all $k \in [\tau-1]$. So, the original name ``base-$(p+1)$'' does not fully reveal the flexibility of the sequential doubly stochastic factorization proposed in this paper.
\par
The following corollary presents the basic properties of the two SDS factors and the $T$-factors.
\begin{corollary}
    The total number of nonzeros in the matrix $T^{(k)}$ is $\nnz (T^{(k)}) = n_k + 2\sum_{i=k+1}^\tau n_i$, for $k \in [\tau]$, and the largest node degree is $d_\mathrm{max} (T^{(k)}) = 2$. In addition,
    \[
        \nnz(A_\mathrm{L}) = \nnz(A_\mathrm{R}) = \sum_{k=1}^\tau (2^k-1) n_k, \qquad d_\mathrm{max} (A_\mathrm{L}) = \tau, \qquad d_\mathrm{max} (A_\mathrm{R}) = 2^{\tau-1}.
    \]
\end{corollary}
\begin{proof}
    The total number of nonzeros in the matrix $T^{(k)}$ and the largest node degree $d_\mathrm{max} (T^{(k)})$ hold from the definition~\eqref{eq:sds-T}.
    \par
    It follows from the definition of $V^{(k)}$~\eqref{eq:sds-V} that 
    \[
        \nnz(V^{(k)}) = n_k + m_k + 2\nnz(V^{(k+1)}), \;\; \text{for all} \ k \in [\tau-1],
    \]
    and $\nnz(V^{(\tau)}) = n_\tau$. Then, recursion over $k$ yields
    \begin{align*}
        \nnz (A_\mathrm{L}) = \nnz (A_\mathrm{R}) = \nnz (V^{(1)}) &= n_1 + m_1 + 2\nnz(V^{(2)}) \\
        &= n_1 + m_1 + 2(n_2 + m_2) + 4\nnz (V^{(3)}) \\ 
        &\;\; \vdots \\ 
        &= \textstyle \sum\limits_{k=1}^{\tau-1} 2^{k-1} (n_k + m_k) + 2^{\tau-1} \nnz (V^{(\tau)}) \\ 
        &= \textstyle \sum\limits_{k=1}^{\tau-1} 2^{k-1} (n_k + m_k) + 2^{\tau-1} n_\tau \\
        &= \textstyle \sum\limits_{k=1}^\tau 2^{k-1} n_k + \sum\limits_{k=1}^{\tau-1} 2^{k-1} m_\tau \\
        &= \textstyle \sum\limits_{k=1}^\tau 2^{k-1} n_k + \sum\limits_{k=1}^{\tau-1} 2^{k-1} \sum\limits_{i=k+1}^\tau n_i \\
        &= \textstyle \sum\limits_{k=1}^\tau 2^{k-1} n_k + \sum\limits_{k=2}^\tau (2^{k-1} - 1) n_k \\
        &= \textstyle \sum\limits_{k=1}^\tau (2^k-1) n_k.
    \end{align*}
    Similarly, the largest node degree of $A_\mathrm{L}$ (and $A_\mathrm{R}$) can be calculated as
    \begin{align*}
        d_\mathrm{max} (A_\mathrm{L}) &= d_\mathrm{max} (V^{(1)}) = d_\mathrm{max} (V^{(2)}) + 1 = \cdots = d_\mathrm{max} (V^{(\tau)}) + \tau - 1 = \tau, \\
        d_\mathrm{max} (A_\mathrm{R}) &= d_\mathrm{max} ((V^{(1)})\tran) = 2 d_\mathrm{max} (V^{(2)})  = \cdots = 2^{\tau-1} d_\mathrm{max} (V^{(\tau)}) = 2^{\tau-1}.   \qedhere
    \end{align*}
\end{proof}

%%%%%%%%%%%%%%%
%%  Section  %%
%%%%%%%%%%%%%%%
\section{Application in decentralized averaging and optimization} \label{sec:app}
In this section, we show how the presented factorizations of the form~\eqref{eq:J-factor} can be used in decentralized averaging (in \Cref{sec:decentr-avg}) and then describe extensions to decentralized optimization (in \Cref{sec:decentr-opt}).
\par
The application scenario considered here involves abstracting agents as high-performance computing (HPC) resources. In modern data centers, machines are organized (or clustered) into racks, and switches and routers are used to connect machines physically. These connections form the so-called \textit{physical topology}, which is often fixed \cite[\S2]{sterling2017high}. In comparison, \textit{virtual topology} refers to the logical network layout between virtual machines (VMs) and other HPC resources. It models how data are transferred between VMs and has the following properties (see, \eg, \cite{bera2017software}).
\begin{itemize}
    \item \textit{Flexibility.} Virtual topology can be dynamically redesigned without reconfiguring the physical hardware. 

    \item \textit{Scalability.} Virtual topology can span multiple physical data centers.
    
    \item \textit{Automation.} Management of virtual topology can be automated via software-defined network (SDN) controllers.
\end{itemize}
These properties are crucial for adapting to evolving traffic demands with the purpose of, \eg, reducing power consumption, network congestion, end-to-end delay, or blocking probability. Hence, with virtual topology, the connection between machines (agents) can be easily manipulated in a dynamic manner. However, the communication cost between machines is based on their physical locations. For example, communication within a network segment (\eg, a rack or a section of the data center) is often more efficient and economical than that between network segments \cite[\S2]{sterling2017high}. The proposed sparse factorization~\eqref{eq:J-factor} of~$J$ offers a robust solution for designing dynamic virtual topologies as it exploits the clustering and hierarchical structure in data centers, leverages the flexibility of virtual topology, and takes into account the non-uniform communication costs within a virtual topology. 

%%%%%%%%%%%%%%%%%%
%%  Subsection  %%
%%%%%%%%%%%%%%%%%%
\subsection{Decentralized average consensus} \label{sec:decentr-avg}
The decentralized average consensus (or decentralized averaging) problem can be formally formulated as follows. In a group of~$n$ agents, each one holds a piece of information, denoted by $x^{(0)}_i \in \reals^d$, and the entire group aims to compute the average $\xbar := \tfrac{1}{n} \sum_{i=1}^n x^{(0)}_i$ via communication. The communication (or connection) between agents is modeled by a sequence of (undirected) graphs (or topologies) $\cG^{(k)} = (\cV, W^{(k)}, \cE^{(k)})$, where $\cV = \{1,\ldots,n\}$ is the node set representing agents, each $\cE^{(k)} \subseteq \cV \times \cV$ is the set of edges (or connections), and the weighted adjacency matrix $W^{(k)} \in \reals^{n \times n}$ stores the weights of the edges. It is assumed that the set of agents remains static while the set of edges can be time-varying. In this context, $W^{(k)}$ is often called the \textit{mixing matrix}, and its entry $w_{ij}^{(k)} \in \reals_{\geq 0}$ applies a weighting factor to the information exchanged between agent~$j$ and agent $i$. If $w_{ij}^{(k)} = 0$, it means agent $i$ is not a neighbor of agent~$j$ in $\cG^{(k)}$; \ie, $(i,j) \notin \cE^{(k)}$. We do not distinguish between a graph $\cG$ and its weight matrix $W$. The state of agent~$i$ (or the information held by~$i$) at iteration~$k$ is designated as $x_i^{(k)}$ and evolves according to the following recursion: for $k \in \natint$,
\[
    x^{(k+1)}_i = \sum_{j \colon (i,j) \in \cE^{(k)}} w_{ij}^{(k)} x^{(k)}_j, \quad \text{for all} \ i \in [n].
\]
The above recursion can be written more compactly as
\begin{equation} \label{eq:avging}
    X^{(k+1)} = W^{(k)} X^{(k)}, \quad \text{where} \ X^{(k)} = \begin{bmatrix} x_1^{(k)} & x_2^{(k)} & \cdots & x_n^{(k)} \end{bmatrix}\tran \in \reals^{n \times d}.
\end{equation}
We say average consensus is achieved if either of the following conditions is satisfied.
\begin{enumerate}
    \item The limit of each $x_i^{(k)}$ is $\xbar$: $\lim_{k \to \infty} x^{(k)}_i = \xbar$ for all  $i \in [n]$.

    \item There exists $\kbar \in \natint$ such that $X^{(\kbar)} = \xbar \ones\tran$ and $X^{(k)} = \xbar \ones\tran$ for all $k \in \natint_{\geq \kbar}$.
\end{enumerate}
\par
In this section, we are interested in modern applications where virtual machines in HPC scenarios are abstracted as agents. In this case, the communication topology can be customized and easily altered during the averaging process. Hence, the proposed sparse factorization of~$J$ helps design topologies that achieve consensus within a finite number of communication rounds. To see this, consider a set of sparse matrices $\{W^{(i)}\}^q_{i=1}$ that satisfies~\eqref{eq:W-factor}. When the associated graph sequence $\{\cG^{(i)}\}_{i=1}^q$ is used as the (time-varying) topologies for decentralized averaging, the iteration~\eqref{eq:avging} yields
\[
    X^{(q)} = W^{(q)} W^{(q-1)} \cdots W^{(2)} W^{(1)} X^{(0)} = \tfrac{1}{n} \ones \ones\tran X^{(0)} = \ones \xbar\tran.
\]
Therefore, unlike classical results where consensus is achieved only asymptotically, \textit{finite-time consensus} (\ie, consensus in \textit{exactly}~$q$ communication rounds) in decentralized averaging is achieved by exploiting sparse factorization of~$J$. Moreover, the sparsity of all the factors in the proposed factorizations of~$J$ helps reduce the per-round communication costs. In decentralized average consensus (as well as decentralized optimization), the communication cost at each round is modeled by either the total number of nonzeros in the mixing matrix $W$ or the largest node degree $d_\mathrm{max}$; see, \eg, the recent book \cite[\S11.3]{ryu2022large}.
\par
To this end, the proposed HB and SDS factorizations of $J$ can be used to construct sparse graph sequences with cheap per-round communication costs and the desirable finite-time consensus property for \textit{arbitrary} number of agents $n \in \natint_{\geq 2}$. This is in contrast to most previous work, which has requirements on the matrix order~$n$ (\eg, $n = p^\tau$ for some $(p,\tau) \in \natint_{\geq 2} \times \natint_{\geq 1}$). Below, we describe in detail how to exploit the factorization $J = J_0 A J_0$ to construct graph sequences $\{W^{(i)}\}$ with finite-time consensus, and then discuss two additional advantages of the proposed HB and SDS factorizations.
\begin{itemize}
    \item \textbf{Phase 1.} The communication network is constructed via a sparse factorization of $J_0 = J_1 \oplus \cdots \oplus J_\tau$. For example, each smaller matrix $J_j \in \reals^{n_j \times n_j}$ can be decomposed as product of $p$-peer hyper-cuboids \cite{nguyen2023graphs}. Then, each mixing matrix in Phase 1 is a direct sum of several $p$-peer hyper-cuboids (and identity matrices).
    
    \item \textbf{Phase 2.} This phase corresponds to the $A$ matrix in~\eqref{eq:J-factor}, which can be the RHB factor, the DSHB factor, the (left or right) SDS factor, or even a sequence of $T$-factors. A detailed comparison between these choices is discussed in the next paragraph and presented in \cref{tab:trade-off}.
    
    \item \textbf{Phase 3.} It corresponds to a sparse factorization of $J_0$, and can be the same as Phase 1.
\end{itemize}
\par
In addition to the ability to handle an arbitrary number of agents, the proposed factorizations (RHB, DSHB, SDS) provide more flexibility to balance the communication costs and the number of communication rounds toward consensus. Recall that the communication cost involved in each iteration~\eqref{eq:avging} is related to the total number of nonzeros and the largest node degree in the communication topology. Thus, using sparser graphs would reduce communication costs but likely increase the total number of iterations toward consensus. For example, using $A_\mathrm{L}$ (or $A_\mathrm{R}$) completes Phase 2 in one iteration, while using the ``one-peer'' $T$-factors in~\eqref{eq:sds-T} results in $\tau-1$ iterations in Phase~2. Such a trade-off is summarized in \cref{tab:trade-off}.
\begin{table}[tb]
    \centering
    \scalebox{0.85}{
    \begin{tabular}{lccccc} \toprule
        Matrices in phase 2 & $A_\mathrm{RHB}$ & $A_\mathrm{DSHB}$ & $A_\mathrm{L}$ & $A_\mathrm{R}$ & $T$-factors \\ \midrule
        Num. of nonzeros & $n + \tau (\tau-1)$ & $\sum_{k=1}^\tau k n_k$ & $\sum_{k=1}^\tau (2^k-1) n_k$ & $\sum_{k=1}^\tau (2^k-1) n_k$ & $n_k + 2\sum_{i=k+1}^\tau n_i$ \\
        Largest node degree $d_\mathrm{max}$ & $\tau$ & $\tau$ & $\tau$ & $2^{\tau-1}$ & 2 \\ 
        Num. of iter. in Phase 2 & 1 & 1 & 1 & 1 & $\tau-1$ \\ \bottomrule
    \end{tabular}
    }
    \caption{Trade-offs between the communication cost (modeled by either the number of nonzeros or the largest node degree $d_\mathrm{max}$) and the number of communication rounds in Phase 2.}
    \label{tab:trade-off}
\end{table}

Moreover, the proposed form of factorization~\eqref{eq:J-factor} handles the issue of non-uniform communication costs mentioned at the beginning of \Cref{sec:app}. In classical decentralized settings, it is typically assumed that the distance between agents is equidistant and that each agent is indistinguishable from another. However, this is not the case in the virtual topology of modern data centers. Recall that \textit{inter-cluster} communication between machines in a rack (or a section of the data center) is often cheaper and swifter than \textit{intra-cluster} communication. Such a characteristic is easily exploited by factorization $J = J_0 A J_0$. Communication in Phases 1 and 3 is all intra-cluster and can be modeled by different sparse factorizations of $J_k$, $k \in [\tau]$. The more expensive inter-cluster communication only happens in Phase~2 and is modeled by the sparse matrix $A$. So, the proposed factorization form~\eqref{eq:J-factor} promotes cheap, intra-cluster communications and limits the more expensive, inter-cluster ones.

\paragraph{Numerical experiments}
Here, we verify that the proposed sparse factors of $J$ satisfy the finite-time consensus property~\eqref{eq:W-factor} in numerical experiments. To do so, we simulate an average consensus problem. Each agent is initialized with a random vector $x_i^{(0)} \sim \cN(0, \Sigma)$ drawn from a Gaussian distribution (with $\Sigma$ symmetric positive definite). The iterates $x_i^{(k)}$ evolve according to the recursion~\eqref{eq:avging}, and the consensus error at each iteration is defined as $\Xi^{(k)} := \tfrac{1}{n} \sum_{i=1}^n \|x_i^{(k)} - \xbar\|^2$.
\par
\Cref{fig:consensus} presents the consensus error using the proposed graph sequences, the one-peer exponential graphs \cite{ying2021exponential} and $p$-peer hyper-cuboids \cite{nguyen2023graphs} for various numbers of agents. (The experimental settings and the presentation style strictly follow from \cite{ying2021exponential,nguyen2023graphs}.) It is known that when $n$ is not a power of~$2$, using the one-peer exponential graphs cannot achieve finite-time consensus \cite{ying2021exponential}. This result, together with \cref{thm:rhb,thm:dshb,thm:sds}, is verified numerically in \cref{fig:consensus}. More specifically, the proposed factorizations have a steep drop in the consensus error, indicating the vanishing of the consensus error, while for one-peer exponential graphs, the consensus error decreases asymptotically. Moreover, note that the number of communication rounds is not the only criterion considered in practice. The per-round communication cost is also important. For example, when $n=43$ (a prime number), the $p$-peer hyper-cuboid reduces to the fully-connected graph and reaches consensus in just one round, while SDS takes 12 rounds. However, the total communication cost (measured by total number of nonzeros; see \cref{tab:trade-off}) for $p$-peer hyper-cuboid is $1849$, compared to only 488 for the SDS factorization.
\begin{figure}
    \centering
    \includegraphics[width=0.9\textwidth]{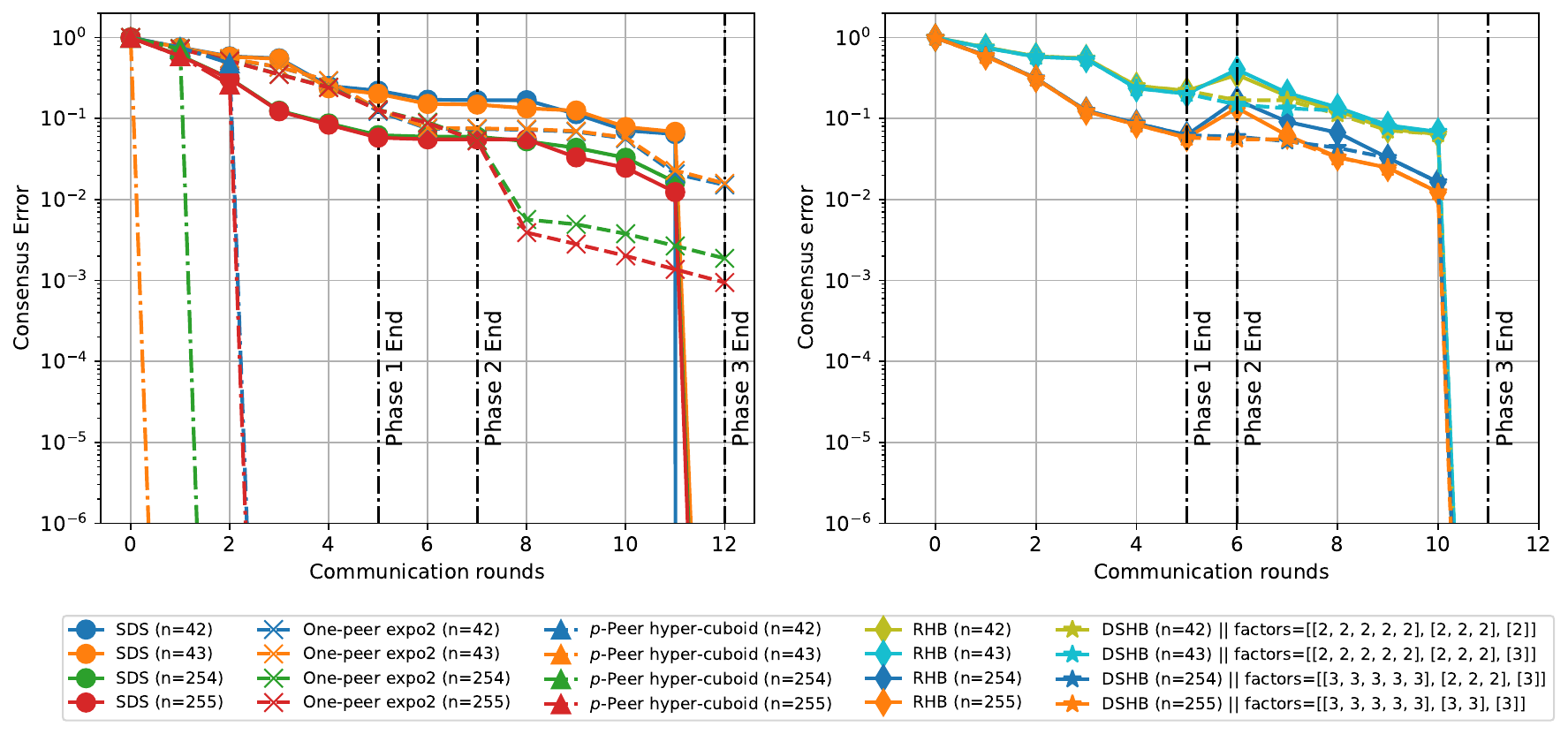}
    \vspace{-10pt}
    \caption{Consensus error versus the number of communication rounds. We examine five graph sequences over four different numbers of agents. The results for SDS factorization, the one-peer exponential graphs, and the $p$-peer hyper-cuboids are presented on the left; the results for RHB and DSHB factorizations are presented on the right. For SDS, RHB, and DSHB, the partition of~$n$ \eqref{eq:n-partition} and the factors used in each cluster are listed at the right-most of legend. For example, $[[2,2,2,2,2], [2,2,2], [2]]$ means the 42 agents are partitioned into three clusters with sizes 32, 8, and 2 and each cluster is further binary partitioned for $J_0$ generation. The factors used for p-peer hyper-cuboid are $42=[2, 3, 7]$, $43=[43]$, $254=[2, 127]$, and $255=[3, 5, 17]$ respectively.
    }
    \label{fig:consensus}
\end{figure}

%%%%%%%%%%%%%%%%%%
%%  Subsection  %%
%%%%%%%%%%%%%%%%%%
\subsection{Decentralized optimization} \label{sec:decentr-opt}
Besides decentralized averaging, sparse factorization of~$J$ is also useful in decentralized optimization. In decentralized optimization, agents collaborate to solve the following optimization problem
\begin{equation} \label{eq:decentr-prob}
    \minimize \ \ f(x) := \frac{1}{n} \sum\limits_{i=1}^{n} f_i (x),
\end{equation}
where the optimization variable is $x \in \reals^d$, and each component function $f_i \colon \reals^d \to \reals$ is continuously differentiable and potentially nonconvex. Each agent $i \in \cV$ only has access to one component function $f_i$, and agents communicate with each other via (time-varying) topologies $\{\cG^{(k)}\}$. It can be shown that the decentralized average consensus problem is a special case of~\eqref{eq:decentr-prob} with $f_i(x) = \tfrac{1}{2} \|x-x_i^{(0)}\|_2^2$.
\par
In the context of decentralized optimization, a sparse factorization of~$J$ offers sequences of graphs that satisfy the finite-time consensus property, and incorporating such graph sequences in decentralized optimization algorithms could significantly reduce the communication cost in the algorithm while achieving a comparable convergence rate (compared with decentralized algorithms using traditional communication protocols) \cite{ying2021exponential,nguyen2023graphs}. Consider, for example, the decentralized gradient descent (DGD) algorithm, an extension of the gradient descent method to the decentralized setting:
\[
    X^{(k+1)} = W^{(k)} X^{(k)} - \alpha \nabla \f (X^{(k)}),
\]
where $\nabla \f(X) := \begin{bmatrix} \nabla f_1(x_1) & \cdots & \nabla f_n(x_n) \end{bmatrix}\tran \in \reals^{n \times d}$ is a matrix of all the component gradient functions and $\alpha \in \reals_{>0}$ is the step size. At each DGD iteration, the matrix $W^{(k)}$ is the weighted adjacency matrix of a graph $\cG^{(k)}$, and it can be sparse factors of $J$. More specifically, in DGD, one iterates from $W^{(1)}$ to $W^{(q)}$ in \eqref{eq:J-factor} and restart with $W^{(1)}$ again. In most existing analyses for DGD (and other decentralized optimization algorithms), the weight matrices $\{W^{(k)}\}$ are assumed to be connected and doubly stochastic (see, \eg, the book \cite[\S11.3]{ryu2022large}). Recall that the $T$-factors \eqref{eq:sds-T} are doubly stochastic yet not connected. So when the $T$-factors are used in DGD, the convergence guarantee is different from classical results. To see this, recall that the convergence rate of DGD with a \textit{connected} static graph is inversely proportional to the so-called \textit{spectral gap} $(1-\rho(W))$ \cite{yuan2016convergence}, where $\rho(W)$ is the second largest eigenvalue (in modulus) of~$W$. For time-varying connected graphs, the convergence rate is inversely proportional to the \textit{worst-case} spectral gap $(1-\rho_\mathrm{max})$ \cite{koloskova2020unified}, where $\rho_\mathrm{max} := \max\{\rho(W^{(k)}) \colon k \in \natint\}$. In contrast, when $T$-factors are used, the convergence rate of DGD follows from \cite[Theorem~1]{ying2021exponential} (because the topology sequence constructed using~\eqref{eq:J-factor} and $T$-factors satisfies all the assumptions stated in \cite{ying2021exponential}) and reads as
\begin{equation} \label{eq:app-rate}
    \frac{1}{K} \sum_{k=1}^K \|\nabla \f(X^{(k)})\|^2 = O \left(\frac{n q^2}{K} \right),
\end{equation}
where $q$ is the finite-time consensus parameter in~\eqref{eq:W-factor}. Unlike the classical convergence results that depend on the spectral gap, the rate~\eqref{eq:app-rate} is independent of the connectivity of any of the individual graphs, but instead considers the joint effect of all the topologies used in the algorithm. Therefore, the proposed factorization enables DGD to handle potentially sparser topologies that can be disconnected during certain iterations. The same conclusion can also be drawn for DGD with momentum \cite{ying2021exponential} and the (decentralized) gradient tracking algorithm \cite{nguyen2023graphs}.
\par
\paragraph{Numerical experiments}
Numerical evidence is presented to demonstrate the potential benefits of using the $T$-factors in decentralized optimization algorithms. We apply DGD to solve the least squares problem (\ie, \eqref{eq:decentr-prob} with each $f_i(x) = \|A_i x - b_i\|^2$). The entries in each $A_i \in \reals^{m \times d}$ are independently and identically distributed (IID) random variables drawn from the standard distribution, and so are the vectors $\{\xtilde_i\}_{i=1}^n \subset \reals^d$. The vector $b_i \in \reals^d$ is then computed by $b_i = A_i \xtilde_i + \delta z_i$, where $\delta \in \reals_{>0}$ is a prescribed constant and $z_i \in \reals^d$ is Gaussian random noise. In the experiments, we set $n=241$ (a prime number), $m=100$, $d=50$, and $\delta=0.1$. We construct the partition~\eqref{eq:n-partition} via binary representation of $n$, use the one-peer hyper-cubes in Phases~1 and 3, and use the proposed $T$-factors \eqref{eq:sds-T} in Phase 2.
\par
\Cref{fig:lr} presents the simulation results for various sparse topologies that have a maximum degree of $1$. We see that DGD with $T$-factors has similar, if not better, performance compared to DGD with other sparse topologies. Yet, only the proposed SDS factorization considers the clustering and hierarchical structures in modern application scenarios. So, despite the similar convergence rate, using the proposed $T$-factors takes fewer inter-cluster communication and has lower total communication cost. An intriguing observation is that the proposed SDS factorization with $T$-factors exhibits oscillation in the mean-square error (MSE) over time, whereas DGD with one-peer exponential graphs approaches a more stable MSE asymptotically. A closer examination finds that the oscillation period equals to $q$, the length of the factorization in~\eqref{eq:J-factor}. This suggests that the oscillation may stem from the three-phase nature of SDS factorization and the presence of some disconnected factors in SDS factorization. Consequently, during each period, agents across different clusters may have different iterates $x_i^{(k)}$ and only achieve consensus at the end of each period.
\begin{figure}
    \centering
    \includegraphics[width=0.82\textwidth]{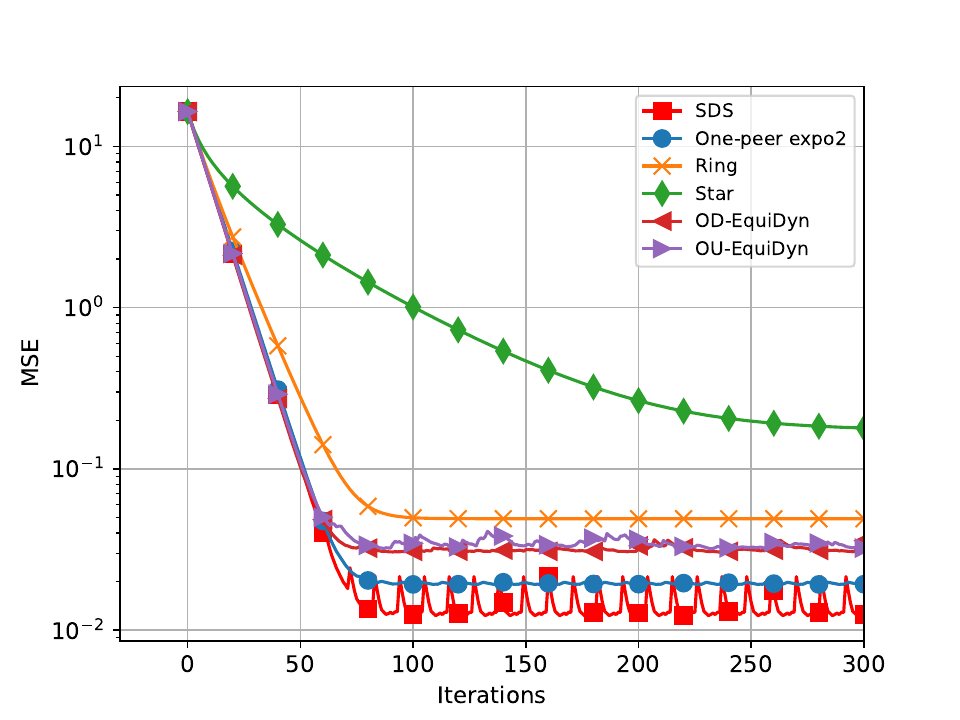}
    \caption{Mean-square error (MSE) versus the number of DGD iterations for the decentralized least squares problem, with various sparse graphs. The convergence of DGD with one-peer exponential graphs is plotted in blue, and that with one-peer undirected EquiTopo graphs \cite{song2023communicationefficient} is in green.}
    \label{fig:lr}
\end{figure}

%%%%%%%%%%%%%%%
%%  Section  %%
%%%%%%%%%%%%%%%
\section{Conclusion}  \label{sec:conclusion}
In this paper, we study the sparse factorization $J = J_0 A J_0$, where $J = \tfrac{1}{n} \ones_n \ones_n\tran$ is the (scaled) all-ones matrix and $J_0 = J_1 \oplus \cdots \oplus J_\tau$ is the direct sum of several smaller (scaled) all-ones matrices. We introduce the hierarchically banded structure of a symmetric matrix, based on which we present two types of hierarchically banded factorization of $J$: the reduced hierarchically banded (RHB) factorization and the doubly stochastic hierarchically banded (DSHB) factorization. Moreover, inspired by the DSHB factorization, we propose the sequential doubly stochastic (SDS) factorization, which further factorizes the matrix $A$ as the product of a sequence of symmetric, doubly stochastic matrices. We then discuss the usefulness of the proposed factorizations in decentralized average consensus and decentralized optimization. The presented three types of sparse factorization offer much flexibility in handling the trade-off between the per-iteration communication cost and the total number of communication rounds in decentralized averaging (and optimization). 
\par
Finally, recall that the partition \eqref{eq:n-partition} is assumed to be given and fixed throughout the paper. Further investigation is needed in the design of this partition to fully leverage the power of the proposed sparse factorizations in decentralized optimization.

\end{document}